\documentclass[11pt, a4paper, fleqn]{amsart}
\usepackage[latin1,utf8]{inputenc}
\usepackage{amsmath}
\usepackage{amsthm}
\usepackage{amsfonts}
\usepackage{amssymb}
\usepackage{amscd,amssymb,amsopn,amsmath,amsthm,graphics,amsfonts,enumerate,verbatim,calc}
\usepackage[T1]{fontenc}
\usepackage[all]{xy}
\usepackage{bbm}
\usepackage{enumitem}
\usepackage{color}
\usepackage[hidelinks]{hyperref}
\usepackage{graphicx}
\usepackage{transparent}
\usepackage{graphics}
\usepackage{xcolor}
\usepackage{color}
\usepackage{geometry}
\usepackage{marginnote}

\usepackage{pgf,tikz}
\usepackage{mathrsfs}
\usetikzlibrary{arrows,decorations.pathmorphing,backgrounds,positioning,fit,petri,cd}
\usepackage{float}

\theoremstyle{plain}
\newtheorem{theorem}{Theorem}
\newtheorem*{theorem*}{Theorem}
\newtheorem{proposition}[theorem]{Proposition}

\theoremstyle{remark}
\newtheorem{remark}[theorem]{Remark}

\theoremstyle{definition}

\newtheorem*{definition*}{Definition}

\newtheorem{defth}[theorem]{Definition/Theorem}

\def\Hom{\mbox{Hom}}
\def\Ext{\mbox{Ext}}
\def\End{\mbox{End}}

\def\add{\mbox{add}\,}
\def\mod{\mbox{mod}\,}

\newcommand{\M}{\mathscr{M}}

\newcommand{\rk}{\mathsf{rk}}

\begin{document}

\title{The size of a stratifying system can be arbitrarily large}

\author{Hipolito Treffinger}
\address{[\textbf{HT}] Institut de Math\'ematiques Jussieu - Paris Rive Gauge, Universit\'e Paris Cit\'e. Paris, France.}
\email{treffinger@imj-prg.fr}

\thanks{The author is supported by the European Union’s Horizon 2020 research and innovation programme under the Marie Sklodowska-Curie grant agreement No 893654. He is also partially funded by the Deutsche Forschungsgemeinschaft (DFG, German Research Foundation) under Germany's Excellence Strategy Programme -- EXC-2047/1 -- 390685813.}
\dedicatory{In honour of Ibrahim Assem on the occasion of his retirement.}

\maketitle

\begin{abstract}
In this short note we construct two families of examples of large stratifying systems in module categories of algebras.
The first examples consists on stratifying systems of infinite size in the module category of an algebra $A$.
In the second family of examples we show that the size of a finite stratifying system in the module category of a finite dimensional algebra $A$ can be arbitrarily large in comparison to the number of isomorphism classes of simple $A$-modules.
We note that both families of examples are built using well-established results in higher homological algebra.
\end{abstract}

\section{Introduction}
In this paper, $A$ is a basic finite-dimensional algebra over an algebraically closed field $K$, $\mod A$ is the category of finitely presented (right) $A$-modules and $K_0(A)$ denotes the Grothendieck group of $A$.

The notion of exceptional sequences originated in algebraic geometry \cite{Bondal, Goro, GoroRuda, RudakovExceptional} and was later introduced to representation theory in \cite{CB, R}, becoming an important subject of study in both disciplines. 

On the representation theory side, the definition of exceptional sequence can be stated in the module category of any finite-dimensional algebra.
Despite this, most of the articles on the subject studied exceptional sequences in the module category of hereditary algebras.
Outside the hereditary case, the notion that have been mostly studied is the more general notion of stratifying systems, firstly introduced in \cite{ES2003} (see also \cite{MMS-SSS, ExtPSS-MMS}).
If the algebra is hereditary, it has been proven in \cite{Cadavid1} that every stratifying system is an exceptional sequence. 
However, this is not true in general. 
For instance, the stratifying systems that we build in this note are not exceptional sequences (see Remarks~\ref{rmk:1} and \ref{rmk:2}).
We now recall the definition of stratifying systems and exceptional sequences. 

\begin{definition*}\label{def:stratsys}
Let $A$ be an algebra.
A \textit{stratifying system} of size $ t \in \mathbb{N} \cup \{\infty\}$ in the category $\mod A$ of finitely generated $A$-modules is a pair $(\Theta, \leq)$ where \mbox{$\Theta := \{ \theta_i: i \in [1,t],  i \neq \infty \}$} is a family of indecomposable objects in $\mod A$ and $\leq$ is a linear order on the set $[1,t]:=\{1, \dots , t\}$ such that $\Hom_A(\theta_j, \theta_i)=0$ if $i < j$ and $\Ext^1_A(\theta_j, \theta_i)=0$ if $i \leq j$.
Moreover, we say that a stratifying system $( \Theta, \leq )$ is an \textit{exceptional sequence} if \mbox{$\End_A(\theta_i) \cong K$} for all $i \in [1,t]$ and $\Ext^n_A(\theta_j, \theta_i) = 0$ for all $i \leq j$ and $n \in \mathbb{N}$.
\end{definition*}

There are numerous works studying the consequences of the existence of a stratifying system in the module category of an algebra, see for instance \cite{Cadavid2, Cadavid1, ES2003, Erdmann2005, MMS-SSS, ExtPSS-MMS}.
However, the existence of stratifying systems in module categories is a problem that has received less attention.
To our knowledge, the only works addressing the existence of exceptional sequences or stratifying systems outside the hereditary case are \cite{Meltzer1995} for canonical algebras,  \cite{HP} for the Auslander algebra of $K[x]/x^t$, \cite{Persson2020} for quotients of type $\mathbb{A}$ zig-zag algebras and \cite{Buan2021, MendozaTreffinger} for arbitrary algebras using techniques from $\tau$-tilting theory.

The classical examples of stratifying systems are the so-called \textit{canonical} stratifying systems.
These are stratifying systems which are constructed using all indecomposable projective modules and, as a consequence, the size $t$ of every canonical stratifying system coincides with $\rk(K_0(A))$. 
Also, it was shown in \cite{Cadavid1} that the size of every stratifying system in the module category of a hereditary algebra $H$ is bounded by $\rk(K_0(H))$.

This is not true for stratifying systems in the module category of an arbitrary algebra since there are examples of stratifying systems whose size is bigger than the rank of the Grothendieck group of the algebra, see for instance \cite[3.2]{ES2003} and \cite[Remark~2.7]{MMS-SSS}. 
However, it was conjectured the existence of an upper-bound for the size of a stratifying system in the module category $A$ which was a function of $\rk(K_0(A))$.

In Section~\ref{sec:infinite}, we show that this conjecture is false by proving the existence of a family of algebras having a stratifying system of infinite size.
Later, in Section~\ref{sec:verylarge}, we construct a family of finite stratifying systems whose size cannot be linearly bounded by the rank of the Grothendieck group of their ambient module category.

\section{stratifying systems of infinite size}\label{sec:infinite}
The notion of $d$-representation infinite algebras was introduced and first studied by Herschend, Iyama and Oppermann in \cite{HIO2014}. 
In this section we use a particular family of $d$-representation infinite algebras, the so-called \textit{Bellinson algebras} \cite[Example~2.15]{HIO2014}, to construct examples of a stratifying systems of infinite size.
We note that the choice of this particular family is made to give an explicit example.
However the same construction of stratifying system of infinite size can be performed in any $d$-representation infinite algebra.

Fix a positive integer $d>1$. 
Then the \textit{Beilinson algebra} $\mathscr{B}_d$ is the path algebra of the quiver 
$$ \xymatrix{
1  \ar@/^0.7pc/[rr]^{a^1_0}_{\scalebox{0.6}{\vdots}}\ar@/_0.7pc/[rr]_{a^1_{d}} & & 2 \ar@/^0.7pc/[rr]^{a^2_0}_{\scalebox{0.6}{\vdots}}\ar@/_0.7pc/[rr]_{a^2_{d}} & & 3} 
\quad\cdots\quad\xymatrix{
d \ar@/^0.7pc/[rr]^{a^{d}_0}_{\scalebox{0.6}{\vdots}}\ar@/_0.7pc/[rr]_{a^{d}_{d}}& & d+1}
$$
modulo the ideal of relations generated by the elements of the form $a_i^{k}a_j^{k+1}-a_j^ka_i^{k+1}$.
Also, we recall that the $d$-Auslander-Reiten translations in $\mod \mathscr{B}_d$ are defined as 
\[ \tau_d (-) := H^0(\nu_d (-) ): \mod \mathscr{B}_d \to \mod \mathscr{B}_d\]
\[ \tau^{-}_d (-) := H^0(\nu^-_d (-) ): \mod \mathscr{B}_d \to \mod \mathscr{B}_d\]
where $\nu_d$ is the autoequivalence of $D^d(\mod \mathscr{B}_d)$ defined as the composition $\nu_d := \nu \circ [-d]$ of the Nakayama functor $\nu : D^d(\mod \mathscr{B}_d) \to D^d(\mod \mathscr{B}_d)$ with the inverse $[-d]$ of the $d$-th suspension functor $[d]: D^d(\mod \mathscr{B}_d) \to D^d(\mod \mathscr{B}_d)$.

It is clear that as a module over itself, $\mathscr{B}_d \cong \bigoplus_{i=1}^{d+1} P(i)$, where every $P(i)$ is indecomposable projective in $\mod \mathscr{B}_d$ and every indecomposable projective $P$ is isomorphic to $P(i)$ for some $1 \leq i \leq d+1$. 
Moreover, since $\mathscr{B}_d$ is basic, we have that $P(i)$ is not isomorphic to $P(j)$ if $i$ is different from $j$.
Likewise, every indecomposable injective is isomorphic to one of the modules $I(i)$, where $D \mathscr{B}_d \cong \bigoplus_{i=1}^{d+1} I(i)$.
Following this notation we define $P_d(i):= \{ \tau^{-j}_d P(i) : j \in \mathbb{N} \}$ and $I_d(i) := \{\tau_d^j I(i) : j \in \mathbb{N}\}$ for every $1 \leq i \leq d+1$.
We note that the natural order $\leq$ in $\mathbb{N}$ induces a natural order in $P_d(i)$ and $I_d(i)$ for every $1\leq i\leq d+1$.

\begin{theorem}
Let $d$ be a positive integer greater than two and let $\leq$ be the natural order in $\mathbb{N}$. 
With the notation above, $(P_d(i), \leq)$ is a stratifying system in $\mod \mathscr{B}_d$ of infinite size for all $1 \leq i \leq d+1$. 
Similarly, $(I_d(i), \leq^{op})$ is a stratifying system $\mod \mathscr{B}_d$ of infinite size for all $1 \leq i \leq d+1$.
\end{theorem}

\begin{proof}
We only prove the case of $(P_d(i), \leq)$ for a given $1 \leq i \leq d+1$, since the other cases are similar. 
First, \cite[Proposition~4.10.(a)]{HIO2014} states that $\tau_d^{-j}P(i)$ and $\tau_d^{-l}P(i)$ are not isomorphic if $j \neq l$.
Hence $P_d(i)$ has infinitely many objects. 

It follows from Propositions 2.3 and 2.9 in \cite{HIO2014} that $\Hom_{\mathscr{B}_d}(\tau_d^{-j} P(i), \tau_d^{-l} P(i))=0$ if $l<j$.
Also, \cite[Proposition~4.10.(f)]{HIO2014} implies that $\Ext^1_{\mathscr{B}_d}(\tau_d^{-j} P(i), \tau_d^{-l} P(i))=0$ for all $j,l \in \mathbb{N}$.
In particular $\Ext^1_{\mathscr{B}_d}(\tau_d^{-j} P(i), \tau_d^{-l} P(i))=0$ if $l \leq j$.
Then $(P_d(i), \leq)$ is a stratifying system of infinite size.
\end{proof}

\begin{remark}\label{rmk:1}
We note that the stratifying system $(P_d(i), \leq)$ is not an exceptional sequence, since $\Ext^d_{\mathscr{B}_d}( \tau_d^{-1} M, M)\neq 0$ for all $M \in P_d(i)$.
Likewise, $(I_d(i), \leq^{op})$ is not an exceptional sequence since $\Ext^d_{\mathscr{B}_d}(M, \tau_d M)\neq 0$ for all $M \in I_d$.
\end{remark}

\section{Stratifying systems for higher Auslander algebras}\label{sec:verylarge}

In the previous section we show the existence of stratifying systems of infinite size.
In this section we show the existence of a family of algebras having finite stratifying systems whose size grows quicker than the rank of the Grothendieck group of the algebras.

\begin{theorem}\label{thm:big}
For every positive integer $m$, there exists an algebra $A_m$ and a stratifying system $(\Theta_m, \leq)$ of size $s_m$ in $\mod A_m$ such that $s_m > m. \rk(K_0(A_m))$.
\end{theorem}

Before proving our theorem, we need to recall some notions and results of higher homological algebra that will be used in our proof. 
A subcategory $\M$ of $\mod A$ is said to be \emph{$d$-cluster tilting} if 
$$\M  = \{X \in \mod A : \Ext_A^i(X,M)=0 \text{ for all $M \in \M$ and $1\leq i \leq d-1$}\}$$
$$ \ \ \ \ \  = \{Y \in \mod A : \Ext_A^i(M,Y)=0 \text{ for all $M \in \M$ and $1\leq i \leq d-1$}\}.$$
An $A$-module $M$ is said to be $d$-cluster tilting if $\add M$ is a $d$-cluster tilting subcategory of $\mod A$.

In general, given an algebra $A$ it is very difficult to know if $A$ has a $d$-cluster tilting subcategory (see \cite{Vaso21} and the references therein).
However, we are interested in the so-called \textit{higher Auslander algebras} introduced in \cite{Iyama2011a}.
These algebras are defined inductively as follows. 

\begin{defth}\cite{Iyama2011a}\label{thm:Iyama}
Let $n$ be a positive integer and $d$ be a non-negative integer. 
The $d$-Aulander algebra $\mathbb{A}_n^d$ is the path algebra of the linearly oriented quiver of type $\mathbb{A}_n$ if $d=0$, or is the endomorphism algebra $\End_{\mathbb{A}_n^{d-1}}(M)$ of the $d$-cluster tilting module $M$ of $\mathbb{A}_n^{d-1}$ for all $d>0$. 
\end{defth}

In fact, there are several combinatorial descriptions of the quiver and relations of $\mathbb{A}_n^d$ \cite{Iyama2011a, OT2012, JKPK, HJ21}. 
In this paper we follow the notation appearing in \cite{OT2012}.

\begin{proposition}\cite[Theorem 3.6]{OT2012}\label{prop:quiver}
Let $\mathbb{A}_n^d$ be the $d$-Auslander algebra of $\mathbb{A}_n$.
Then there is a bijection between the vertices of the quiver of $\mathbb{A}_n^d$ with the set 
$$V_{n,d}= \left\{(x_0, x_1, \dots, x_d) : 1\leq x_0 < x_1 < \dots < x_d \leq n+d  \right\}.$$
Moreover, given two elements $\underline{x}, \underline{y} \in V_{n,d}$ there is an arrow $\underline{x} \to \underline{y}$ if there exists $0 \leq k \leq d$ such that $y_k=x_k+1$ and $y_i=x_i$ for all $i \neq k$.
\end{proposition}

\begin{remark}\label{rmk:number}
Note that it follows from the previous proposition that the number of vertices of the quiver of $\mathbb{A}_n^d$ is equal to the number of integer lattice points inside the canonical $d$-simplex in $\mathbb{R}^{d+1}$ generated by the interval $[0,n]$.
\end{remark}

\begin{proof}[Proof of Theorem~\ref{thm:big}.]
Let $\mathbb{A}_n^1$ be the ($1$-)Auslander algebra of $\mathbb{A}_n$. 
It follows from Theorem~\ref{thm:Iyama} that $\mathbb{A}_n^1$ has a $2$-cluster tilting module $M= \bigoplus_{i=1}^{t_n} M_i$ such that $\End_{\mathbb{A}_n^1}M = \mathbb{A}_n^2$. 
We claim that there exists  of an order $\leq$ in the set $[1,t_n]:=\{1, \dots, t_n\}$ such that $(\Theta_n, \leq):=(\{M_i : i \in [1,t_n]\}, \leq)$ is a stratifying system.

We first note that it follows from Proposition~\ref{prop:quiver} that the quiver of $\mathbb{A}_n^2$ is an acyclic quiver.
As a consequence, there exists a total order $\leq$ in $[1,t_n]$ such that $\Hom_{\mathbb{A}_n^1}(M_i, M_j)=0$ if $j < i$.
We also have that $\Ext^1_{\mathbb{A}_n^1}(M_i, M_j)=0$ for all $i, j \in [1,t_n]$ because $M$ is a $2$-cluster tilting object.
In particular, $\Ext^1_{\mathbb{A}_n^1}(M_i, M_j)=0$ if $i \leq j$.
Then $(\Theta_n, \leq)$ is a stratifying system.

Now, it is easy to see that $\rk(K_0(\mathbb{A}_n^1)) = \frac{n(n+1)}{2}$.
Moreover the size $t_n$ of $(\Theta_n, \leq)$ is equal to the tetrahedral number $t_n=\frac{n(n+1)(n+2)}{6}$, see Remark~\ref{rmk:number}.
Hence the ratio $R(n)$ between the size of the stratifying system $(\Theta_n, \leq)$ and $\rk(K_0(\mathcal{A}_{n}))$ is 
$$	R(n)	=	 \frac{n(n+1)(n+2)/6}{n(n+1)/2}	 =	\frac{n+2}{3}.$$

In particular, for each $m\in \mathbb{N}$, if we fix $n=3m-1$ we have that size $t_{3m-1}$ of the stratifying system \mbox{$(\Theta_{3m-1}, \leq)$}  in $\mathbb{A}_{3m-1}^1$ is greater than $m.\rk(K_0(\mathbb{A}_{3m-1}^1))$. 
Indeed, 
$$t_{3m-1}= \frac{3m-1+2}{3}. \rk(K_0(\mathbb{A}_{3m-1}^1))= \left(m + \frac{1}{3} \right)  .\rk(K_0(\mathbb{A}_{3m-1}^1)) > m.\rk(K_0(\mathbb{A}_{3m-1}^1)).$$
The result follows by taking $A_m = \mathbb{A}_{3m-1}^1$ and $s_m=t_{3m-1}$.
\end{proof}

\begin{remark}
In the previous proof we fixed $d=1$ to construct our example, but it is easy to see that a similar argument is valid for every $d$.
Indeed, by Propositon~\ref{prop:quiver} we have that the quiver of $\mathbb{A}_{n}^{d}$ is acyclic for every $d$.
Hence we can construct a stratifying system in $\mathbb{A}_{n}^{d}$ for all $n$ and $d$ by means of its $(d+1)$-cluster tilting module.

Moreover, the number of vertices of the quiver of $\mathbb{A}_{n}^{d}$ grows as $n^d$ as $n$ goes to infinity. 
Since there is a stratifying system in $\mod \mathbb{A}_{n}^{d}$ of size $K_0(\mathbb{A}_{n}^{d+1})$, we conclude that for every $m\in \mathbb{N}$ there are infinitely many algebras $\mathcal{A}_m$ with a stratifying system $(\Theta_m, \leq)$ such that the size of $\Theta_m$ is greater than $m.\rk(K_0(\mathcal{A}_m))$.
\end{remark}

\begin{remark}\label{rmk:2}
We note that $(\Theta_m, \leq)$ is not an exceptional sequence since we can find objects $M_i, M_j \in \Theta_m$ such that $\Ext^2_{A_m}(M_i, M_j) \neq 0$ for all $m \in \mathbb{N}$.
\end{remark}

\begin{remark}
A stratifying system $(\Theta, \leq)$ is said to be complete if there is no indecomposable module $N$ such that $(\Theta \cup \{N\}, \leq')$ is a stratifying system such that $i\leq' j$ if $i \leq j$.
It follows from the construction of $(\Theta_m, \leq)$ that there is no indecomposable $A_m$-module $M_0$ such that $(\{M_0\}\cup\Theta_m, \leq)$, where $\leq$ is the natural order in $[0,t_m]$.
Likewise, there is no indecomposable $M_{t_m+1}$ such that $(\Theta_m \cup \{M_{t_m+1}\}, \leq)$ is a stratifying system where $\leq$ is the natural order in $[1, t_m+1]$.
However, is not clear that $(\Theta_m, \leq)$ is a complete stratifying system in general.
\end{remark}

\noindent
{\bf Acknowledgements:} The author would like to thank Octavio Mendoza, Corina Saenz and Aran Tattar for their comments and remarks. 
A special thanks goes to the anonymous referee for their careful reading and helpful suggestions.
He is specially grateful of Gustavo Jasso for pointing out that the arguments of Section~\ref{sec:verylarge} could be used to construct a stratifying system of infinite size, leading to the inclusion of Section~\ref{sec:infinite}.

\def\cprime{$'$} \def\cprime{$'$}

\end{document}